\documentclass[11pt]{amsart}
\usepackage{amsmath, amscd, amsthm, amssymb, graphics, mathrsfs, setspace,fancyhdr,geometry,tabularx,shapepar, hyperref, xcolor, tikz, cite, booktabs, dsfont,amsfonts, marvosym, amsbsy, bm, tikz, array, mathtools,arydshln,tikz-cd}
\usetikzlibrary{matrix,arrows,backgrounds}
\usepackage[all]{xy}
\usepackage{leftidx}

\usepackage{mathrsfs}
\usepackage{wrapfig}
\usepackage{listings}

\usepackage{framed}

\usepackage{fancyhdr}

\usepackage[OT2,T1]{fontenc}
\DeclareSymbolFont{cyrletters}{OT2}{wncyr}{m}{n}
\DeclareMathSymbol{\Sha}{\mathalpha}{cyrletters}{"58}
\DeclareMathSymbol{\Zhe}{\mathalpha}{cyrletters}{17}
\usepackage{wasysym}

\hypersetup{colorlinks=false}
\theoremstyle{plain}
\newtheorem{theorem}{Theorem}[section]

\newtheorem{lemma}[theorem]{Lemma}

\newtheorem{proposition}[theorem]{Proposition}

\newtheorem{question}[theorem]{Question}

\newtheorem{conjecture}[theorem]{Conjecture}
\newtheorem*{theorem*}{Theorem}
\newtheorem*{problem*}{Problem}
\newtheorem*{question*}{Question}
\newtheorem{theoremintro}{Theorem}

\theoremstyle{definition}
\newtheorem{definition}[theorem]{Definition}

\newtheorem{remark}[theorem]{Remark}

\numberwithin{equation}{section}

\renewcommand{\O}{{\mathcal O}}
\renewcommand{\hom}{\operatorname{Hom}}
\newcommand{\real}{{\mathbb R}}

\newcommand{\Z}{{\mathbb Z}}

\newcommand{\op}[1]{\operatorname{#1}}
\newcommand{\Ob}{\operatorname{Ob}}

\newcommand{\gm}{\mathbb{G}_{m}}

\newcommand{\Br}{\operatorname{Br}}
\newcommand{\GL}{\operatorname{GL}}

\newcommand{\Hom}{\operatorname{Hom}}

\newcommand{\Aut}{\operatorname{Aut}}

\newcommand{\bbP}{\mathbb{P}}

\newcommand{\leftexp}[2]{{\vphantom{#2}}^{#1}{#2}}
\newcommand{\weezer}{\leftexp{=}{\kern-0.23em\mathsf{W}}^{\kern-0.21em =}}
\newcommand*{\sheafhom}{\mathrm{H}\kern -.5pt om} 
\newcommand{\cd}[1]{\mathscr{D}(#1)}
\newcommand{\twist}[2]{{}^{#2} \kern -2.5pt #1}
\newcommand{\ICP}{\Zhe}

\begin{document}

\title[Consequences of Exceptional Collections in Arithmetic and Rationality]{Consequences of the Existence of Exceptional Collections in Arithmetic and Rationality}

\author[Ballard]{Matthew R Ballard}
\address{Department of Mathematics, University of South Carolina, 
Columbia, SC 29208}
\email{ballard@math.sc.edu}
\urladdr{\url{http://www.matthewrobertballard.com}}

\author[Duncan]{Alexander Duncan}
\address{Department of Mathematics, University of South Carolina, 
Columbia, SC 29208}
\email{duncan@math.sc.edu}
\urladdr{\url{http://people.math.sc.edu/duncan/}}

\author[Lamarche]{Alicia Lamarche}
\address{Yau Mathematical Sciences Center, Tsinghua University, 
Beijing, China 100084}
\email{lamarche@mail.tsinghua.edu.cn}
\urladdr{\url{http://www.alicialamarche.com/}}

\author[McFaddin]{Patrick K. McFaddin}
\address{Department of Mathematics, Fordham University, 
New York, NY 10023}
\email{pkmcfaddin@gmail.com}
\urladdr{\url{http://mcfaddin.github.io/}}

\begin{abstract}
  A well-known conjecture of Orlov asks whether the existence of a full exceptional collection implies rationality of the underlying variety. We prove this conjecture for arithmetic toric varieties over general fields. We also investigate a slight generalization of this conjecture, where the endomorphism algebras of the exceptional objects are allowed to be separable field extensions of the base field. We show this generalization is false by exhibiting a geometrically rational, smooth, projective threefold over the the field of rational numbers that possesses a full \'etale-exceptional collection but not a rational point. The counterexample comes from twisting a non-retract rational variety with a rational point and full \'etale-exceptional collection by a torsor that is invisible to Brauer invariants. Along the way, we develop some tools for linearizing objects, including a group that controls linearizations.
\end{abstract}

\maketitle
\addtocounter{section}{0}



\section{Introduction}

Developments over the past forty years have established derived categories of coherent sheaves as a versatile language for capturing deep but obscure geometric connections between algebraic varieties. Central to these investigations has been the tie between rationality questions and derived categories. A basic motivating question is the following: 

\begin{question} To what extent can the derived category be used as tool to understand the rationality of a variety? 
\end{question}

Examples in low dimension provide some insight. For a smooth projective curve $C$ over a field $k$, the bounded derived category $\op{D^b}(C) = \op{D^b}( \op{coh} C)$ of coherent sheaves on $C$ admits a full $k$-exceptional (or \'etale-exceptional; see Definition~\ref{def:exceptional}) collection if and only if $C \cong \mathbb{P}^1_k$. 

Over a perfect field $k$, the derived category of a smooth rational projective surface always has a full \'etale-exceptional collection, though not a full $k$-exceptional collection in general. This follows from the classification of minimal rational surfaces; see for example \cite{MTsurvey}, and a case-by-case analysis for del Pezzo varieties \cite{AB}.

It is expected that rationality of $X$ should guarantee that $\op{D^b}(X)$ admits a semi-orthogonal decomposition into components that are not ``too complicated''. The structure of derived categories of Fano threefolds over $\mathbb{C}$ provides evidence for this belief \cite{Kuzsurvey}. Kuznetsov's conjecture on the rationality of a cubic fourfold also follows along this general belief \cite{Kuz4fold}.

In the other direction, Vial showed that any geometrically rational smooth projective surface with a full (numerical) $k$-exceptional collection is $k$-rational \cite{Vial}. Brown and Shipman showed that 
a smooth complex projective surface with a full strong exceptional collection of line bundles is rational \cite{BS}. 

More generally, a conjecture of Orlov states that a smooth projective variety with a full exceptional collection is rational. Even stronger, Lunts conjectures that over a general field $k$, a full $k$-exceptional collection for $X$ implies that $X$ admits a locally-closed stratification into subvarieties that are each $k$-rational \cite{ElaginLunts}. Orlov's Conjecture has been settled affirmatively for surfaces of small Picard rank which admit exceptional collections of line bundles \cite{Zhang}.

This article focuses on Orlov's Conjecture, both in its original form and a slight generalization. The first main result is that Orlov's Conjecture is true for toric varieties over general fields $k$ (also called \emph{arithmetic toric varieties}). Unlike the case of $k = \mathbb{C}$, such varieties need not be rational, nor even retract rational, in general. 

\begin{theoremintro} \label{theorem:FEC_rat}
  Let $X$ be a smooth projective toric variety over a field $k$ with $X(k) \neq \emptyset$. If $\op{D^b}(X)$ has a full $k$-exceptional collection then $X$ is $k$-rational. 
\end{theoremintro}

We then turn to investigating a slight generalization of Orlov's Conjecture: if a geometrically irreducible variety $X$ possesses a full \'etale-exceptional collection, is it necessarily rational?
Here we say an object $E \in \op{D^b}(X)$ is \emph{\'etale-exceptional} if its endomorphism algebra is a separable field extension of $k$ (concentrated in degree 0). The following results show that the existence of a full \'etale-exceptional collection says very little about the rationality of a variety. 

\begin{theoremintro} \label{theorem:Kuznetsov_no}
 There exists a smooth threefold $X$ over $\mathbb{Q}$ that is
geometrically rational, has a $\mathbb{Q}$-point, and
whose derived category admits a semi-orthogonal decomposition 
 into derived categories of smooth points,
but $X$ is not $k$-rational.
\end{theoremintro}

\begin{theoremintro} \label{theorem:BB_no}
 There exists a smooth  
 threefold $Y$ over $\mathbb{Q}$
that is geometrically rational, and
whose derived category admits a semi-orthogonal decomposition 
 into derived categories of smooth points,
 but $Y$ has no $\mathbb{Q}$-points
\end{theoremintro}

The varieties $X$ and $Y$ in the above theorems above are \'etale twisted forms of one
another, i.e., they become isomorphic after extending scalars to a finite separable extension of $k$.
We construct $X$ as a neutral $T$-toric variety where $T$ is a
torus that is not retract rational.
Consequently, the invariant $\ICP(\mathbb{Q},T)$ introduced in \cite{BDLM}
is nontrivial.
The variety $Y$ is then constructed as the twist $\twist{X}{U}$
by a non-trivial torsor $U \in \ICP(\mathbb{Q},T)$.

En route to these theorems, we investigate the twisting of objects of derived categories under the presence of group actions in Section~\ref{section:linearizing_stuff}. These results provide the tools necessary to descend \'etale-exceptional collections to twisted forms using linearizations or equivariant structures. For a general object $E$ with connected automorphisms, we introduce a group $\widetilde{G}_E$ which controls whether $E$ can be linearized; see Definition~\ref{defn:Gtilde} and Lemma~\ref{lem:linearization=section}. We show that if $E$ is \'etale-exceptional and the group is connected, then there is always some $r \geq 1$ such that $E^{\oplus r}$ admits a linearization, extending a result of Polishchuk \cite[Lemma 2.2]{Polishchuk_unity}. Additional results on the representability of $\widetilde{G}_E$ are given in Appendix~\ref{appendix:representability}.

In Section~\ref{section:TCI}, we introduce a class of exceptional collections on toric varieties that are guaranteed to descend to \'etale-exceptional objects on any toric variety with a $k$-point. The proofs of the main results are then included in Section~\ref{section:mainproofs}.


\subsection*{Acknowledgments}

The authors would like to thank B.~Antieau and B.~Kunyavski for several helpful comments. We also are grateful for a conversation with A. Kuznsetov. We thank the anonymous referee for a careful reading and helpful comments.
Via the first author, this material is based upon work supported by the National
Science Foundation under Grant No.~NSF DMS-1501813.
Via the second author, this work was supported by a grant from the Simons Foundation
(638961, AD). The third author was partially supported by a USC SPARC grant. The fourth author was partially supported by an AMS-Simons travel grant and a Fordham Faculty Research Grant.


\section{Background} 

\subsection{Derived categories, semi-orthogonal decompositions, and exceptional collections}

We give some conventions for semiorthogonal decompositions and exceptional collections. 
For a triangulated category $\mathsf{T}$, we use the standard notation $\text{Ext}^n_{\mathsf{T}}(A, B) = \hom _{\mathsf{T}} (A, B[n])$.  For objects $A, B $ of $\op{D^b}(X)$, we use  $\text{End}_X(A)$ and $\text{Ext}_X^n(A, B)$ to denote $\text{End}_{\op{D^b}(X)}(A)$ and $\text{Ext}^n_{\op{D^b}(X)}(A, B)$, respectively.

\begin{definition}[see \cite{BK}]
Let $\mathsf{T}$ be a triangulated category.  A full triangulated subcategory of $\mathsf{T}$ is \emph{admissible} if its inclusion functor admits left and right adjoints.  A \emph{semiorthogonal decomposition} of $\mathsf{T}$ is a sequence of admissible subcategories $\mathsf{C}_1, ..., \mathsf{C}_s$ such that 
\begin{enumerate}
\item $\hom _{\mathsf{T}}(A_i, A_j) = 0$ for all $A_i \in \Ob (\mathsf{C}_i)$, $A_j \in \Ob (\mathsf{C}_j)$ whenever $i > j$.
\item For each object $T$ of $\mathsf{T},$ there is a sequence of
morphisms $0 = T_s \to \cdots \to T_0 = T$ such that the cone of $T_i \to T_{i-1}$ is an object of $\mathsf{C}_i$ for all $i = 1,..., s$.
\end{enumerate}
We use  $\mathsf{T} = \langle \mathsf{C}_1,..., \mathsf{C}_s\rangle$ to denote such a decomposition.
\end{definition}

Particularly nice examples of semiorthogonal decompositions are given by exceptional collections, the study of which goes back to Beilinson \cite{Beilinson}.

\begin{definition}\label{def:exceptional}
Let $\mathsf{T}$ be a $k$-linear triangulated category and let $A$ be a finite-dimensional $k$-algebra of finite homological dimension. An object $E$ in $\mathsf{T}$ is $A$-\emph{exceptional} if the following conditions hold:
\begin{enumerate}
\item $\text{End}_{\mathsf{T}}(E) \cong A$.
\item $\text{Ext}^n_{\mathsf{T}}(E, E) = 0$ for $n \neq 0$.
\end{enumerate}
We say $E$ is \emph{exceptional} if it is $A$-exceptional for a division
algebra $A$. We say $E$ is \emph{\'etale-exceptional} if $A$ is a finite separable field extension of $k$.

A totally ordered set $\mathsf{E} = \{E_1, ..., E_s\}$ of exceptional objects is an \emph{exceptional collection} if $\text{Ext}^n_{\mathsf{T}}(E_i, E_j) = 0$ for all integers $n$ whenever $i >j$.  An exceptional collection is $\emph{full}$ if it generates $\mathsf{T}$, i.e., the smallest thick subcategory of $\mathsf{T}$ containing $\mathsf{E}$ is all of $\mathsf{T}$.  An exceptional collection is \emph{strong} if $\text{Ext}^n_{\mathsf{T}}(E_i, E_j) = 0$ whenever $n \neq 0$.  An \emph{exceptional block} is an exceptional collection $\mathsf{E} = \{ E_1, ..., E_s\}$ such that $\text{Ext}^n_{\mathsf{T}}(E_i, E_j) = 0$ for every $n$ whenever $i \neq j$. An exceptional collection is \emph{\'etale-exceptional} if each of its objects is \'etale-exceptional. A collection is $k$-\emph{exceptional} if each object is $k$-exceptional.
\end{definition}

\subsection{Castravet and Tevelev's Exceptional Collection}
\label{subsection:a3}

In \cite{CT}, Castravet and Tevelev constructed a highly symmetric 
exceptional collection for the toric variety associated
to the fan of Weyl chambers of the root system $A_3$. Forms of this toric variety will be used to establish Theorem~\ref{theorem:BB_no}. 

The most symmetric way to represent the fan of $X(A_3)$ is to write it 
in the the quotient space $\mathbb{Z}^4/\mathbb{Z}(e_0+e_1+e_2+e_3)$. 
This carries a natural action by $S_4$ given by permutation of the 
indices. The maximal cones are the $S_4$-orbit of the cone spanned by 
the rays
\begin{displaymath}
  e_0, e_0+e_1, e_0+e_1+e_2.
\end{displaymath}
The automorphism group of the fan, which we denote $\Sigma_4$, is 
isomorphic to $S_4 \times C_2$ where the additional $C_2$ acts by 
complement on the set of indices in a sum: 
\begin{displaymath}
  \sum_{I \subsetneq \{0,1,2,3\}} e_i \mapsto \sum_{I^c} e_i. 
\end{displaymath}

In \cite{CT}, Castravet and Tevelev construct full $\op{Aut}(A_n)$-stable exceptional collections of sheaves for each of the split toric varieties corresponding to the root systems of type $A$, denoted $X(A_n)$. We recall those now.

An important idea in the construction are the cuspidal pieces of the derived categories of the varieties $X(A_n)$. An object $F$ of $\op{D^b}(X(A_n))$ is called \emph{cuspidal} if for all sub-root systems $A_{\ell} \leq A_n$ of type $A$, we have 
\begin{displaymath}
  \mathbf{R} \pi_\ast F = 0
\end{displaymath}
where $\pi : X(A_n) \to X(A_{\ell})$ is the corresponding map of toric varieties. 

The collections constructed in \cite{CT} are built inductively by pulling back the cuspidal pieces from subsystems of type $A_{\ell}$ for $\ell < n$ and then adding in the cuspidal part for $n$. We recall these collections in low dimensions.

\begin{itemize}
 \item $X(A_0) = \op{Spec} k$. The collection and whole cuspidal piece is $\mathcal O$. 

 \item $X(A_1) = \mathbb{P}^1$. The collection in \cite{CT} is $\{\mathcal O(-1), \mathcal O\}$ and the cuspidal piece is $\mathcal O(-1)$. 

\item $X(A_2)$ is del Pezzo surface of degree $6$. Viewing $X(A_2)$ as the blowup of $\bbP ^2$ at 3 non-colinear points, let $H$ be the pullback of the
hyperplane divisor on $\bbP^2$ and $E_i$ the exceptional divisors, $i =
1, 2, 3$. Then the collection is given by
\begin{displaymath}
  \{ \O(-H), \O(-2H + E_1 + E_2 + E_3), \O(-H + E_1),
  \end{displaymath}
  \begin{displaymath}
   \O(-H + E_2), \O(-H + E_3), \mathcal O\}.
\end{displaymath}
The cuspidal part is $\O(-H), \O(-2H + E_1 + E_2 + E_3)$. The line bundles $\O(-H + E_1), \O(-H + E_2), \O(-H + E_3)$ are $\O(-1)$ pulled back from the three copies of $A_1$ in $A_2$, and of course $\mathcal O$ is pulled back from $A_0$. 

 \item For $X(A_3)$, the collection consists of $1$ line bundle pulled back from $X(A_0)$, $6$ lines bundles coming from pulling back $\mathcal O(-1)$ from the six copies of $A_1$ in $A_3$, and $4 \times 2 = 8$ line bundles coming from pulling back $\O(-H), \O(-2H + E_1 + E_2 + E_3)$ from the four copies of $A_2$ in $A_3$, together with the cuspidal part. 

The cuspidal part breaks up into two blocks: one consisting of $3$ line bundles and the other of $6$ torsion sheaves. The line bundles are pulled back from the embedding of $X(A_3)$ into the wonderful compactification of the adjoint form of $A_3$ as the closure of the maximal torus. 

The torsion pieces can be described as follows. The divisors of $X(A_3)$ are the weights of $A_3$. The orbits are in bijection with nodes in the Dynkin diagram. There are six divisors corresponding to the middle node. Each such divisor, as a toric variety, is isomorphic to $X(A_1 \times A_1) \cong X(A_1) \times X(A_1)$. The torsion block consists of the $i_\ast \mathcal O(-1,-1)$ for each middle weight. 
\end{itemize}

\subsection{Arithmetic toric varieties} 
\label{subsection:arithmetic_toric_varieties}

We recall the theory of toric varieties defined over arbitrary fields.
These varieties have been treated in
\cite{Duncan, ELFST, MerkPan, Vos82Projective, VosKly},
and are sometimes called \emph{arithmetic toric varieties}.

Let $k$ be a field, $\bar{k}$ its separable closure, and $\Gamma = \op{Gal}(\bar{k}/k)$ its absolute Galois group.
An \emph{\'etale algebra} over $k$ is a direct sum
$E = F_1 \oplus \cdots \oplus F_r$
where $F_1, \ldots, F_r$ are separable field extensions of $k$.
There is an antiequivalence between finite continuous $\Gamma$-sets $\Omega$ and
\'etale algebras $E$ via
\[
\Omega = \Hom_{k-\mathsf{Alg}}(E,\overline{k})
\textrm{ and } E = \Hom_{\Gamma-\mathsf{Set}}(\Omega,\overline{k})
\]
with the natural $\Gamma$-action and $k$-algebra structure
on $\overline{k}$ (see, e.g., \cite[\S{18}]{BOI}).

\begin{definition}
A $k$-\emph{torus} is an algebraic group $T$ over $k$ such that $T _{\overline{k}} \cong \gm ^n$. A torus is \emph{split} if $T \cong \gm ^n$.
A $\Gamma$-lattice is a free finitely generated abelian group with
a continuous action of $\Gamma$.
A $\Gamma$-lattice is \emph{permutation} if it has a basis permuted by
$\Gamma$.
\end{definition}

Recall that there is an anti-equivalence of categories between
$\Gamma$-lattices and $k$-tori, which we call \emph{Cartier duality}.
Given a torus $T$, the Cartier dual (or \emph{character lattice}) $\widehat{T}$ is the
$\Gamma$-lattice $\op{Hom}(\overline{T}, \mathbb{G}_{m,\overline{k}})$.
Given a $\Gamma$-lattice $M$, we use the notation $\cd M$ for the
Cartier dual torus.

Given an \'etale algebra $E$, the Weil restriction $R_{E/k} \gm$ is
isomorphic to the torus $T$ whose character lattice $\widehat{T}$
is a permutation module with a basis given by the $\Gamma$-set
$\Omega$ corresponding to $E$.

An important class of tori arise via the kernel of the 
norm map. Recall that any \'etale algebra $E$ over $k$ has a \emph{norm map}
$N : E^\times \to k^\times$.  We obtain an exact sequence
\[
1 \to R_{E/k}^{(1)} \gm \to R_{E/k} \gm \to \gm \to 1
\]
of tori over $k$, where the torus $R_{E/k}^{(1)} \gm$ is called
the \emph{norm-one torus} of the extension $E/k$. 
Of particular importance to this article is the norm-one torus for 
an extension $K/k$ with 
\begin{displaymath}
  \op{Gal}(K/k) = C_2 \times C_2. 
\end{displaymath}
We can describe the Cartier dual of $R^{(1)}_{K/k} \mathbb{G}_m$ 
explicitly. We have the short exact sequence 
\begin{displaymath}
  0 \to \mathbb{Z} \to \mathbb{Z}[C_2 \times C_2] \to J \to 0
\end{displaymath}
where the first map is determined by 
\begin{displaymath}
  1 \mapsto 1 + \sigma + \tau + \sigma\tau 
\end{displaymath}
if we write $C_2 \times C_2 = \langle \sigma \rangle \times 
\langle \tau \rangle$. Note that $J$ is dual to the kernel $I$ of the augmentation of $\mathbb{Z}[C_2 \times C_2]$. 

\begin{lemma} \label{lemma:aug_ideal_presentation}
  There is a basis for $I$ where $\sigma$ and $\tau$ act via the matrices 
  \begin{displaymath}
    \begin{pmatrix}
      -1 & 0 & 0 \\
      0 & 0 & -1 \\
      0 & -1 & 0 \\
    \end{pmatrix}
    \text { and }
    \begin{pmatrix}
      1 & 0 & 0 \\
      1 & -1 & 0 \\
      1 & 0 & -1 \\
    \end{pmatrix}
  \end{displaymath}
  respectively.
\end{lemma}

\begin{proof}
  We take the basis $1 - \sigma$, $\tau -1$, and $\sigma - \sigma \tau$. We leave it to the reader to check that multiplication by $\sigma$ induces the first matrix while multiplication by $\tau$ gives the second. 
\end{proof}

\begin{definition}\label{defn:ATV}
Given a torus $T$, a \emph{toric} $T$-\emph{variety} is a normal variety with a faithful $T$-action and a dense open $T$-orbit.  A toric $T$-variety is \emph{split} if $T$ is a split torus.
A toric $T$-variety whose dense open $T$-orbit contains a $k$-rational
point is called \emph{neutral} \cite{Duncan} (or a \emph{toric}
$T$-\emph{model} \cite{MerkPan}).
An orbit of a split torus always has a $k$-point, so a split toric
variety is neutral; but the converse is not true in general.
\end{definition}

To check if a toric variety has a rational point, it suffices to check on the open dense orbit. 

\begin{proposition} \label{proposition:rational_points}
  Let $X$ be a smooth projective toric $T$-variety over a field $k$ 
  with dense open $T$-orbit $U$. Then $X(k) \not = \emptyset$ if and 
  only if $U(k) \not = \emptyset$.
\end{proposition}

\begin{proof}
See \cite[Proposition 4]{VosKly}.
\end{proof}

A \emph{toric variety} $X$ is simply a toric $T$-variety for some choice
of torus $T$.
Note that two toric varieties may be isomorphic as varieties,
but the isomorphism may not respect the torus actions.
Additionally, there may be multiple non-isomorphic tori
giving the same variety the structure of a toric variety.
However, any $k$-form of a toric variety is a toric variety
(albeit for a potentially different torus action).
Thus, understanding $k$-forms of toric varieties can be subtle.
We recall some tools for understanding forms of toric varieties;
a more comprehensive account can be found in \cite{Duncan}.

Let $\Sigma$ be a smooth projective fan in the space $N \otimes \real$
associated to the lattice $N$.
let $X(\Sigma)$ be the corresponding split toric variety.
Let $\Aut(\Sigma)$ be the group of isomorphisms of the fan $\Sigma$;
in other words,
$\Aut(\Sigma)$ is the subgroup of elements $g \in \operatorname{GL}(N)$
such that $g(\sigma) \in \Sigma$ for every cone $\sigma \in \Sigma$.
There is a natural inclusion of $\gm^n \rtimes \Aut(\Sigma)$
into $\Aut(X)$ as the subgroup leaving the open orbit $\gm^n$-invariant.

The natural map 
\[ H^1(k, T \rtimes \Aut(\Sigma)) \to H^1(k, \Aut(X)) \]
in Galois cohomology is surjective;
the failure of this map to be a bijection amounts to the fact that there
may be several non-isomorphic toric variety structures
on the same variety.

Suppose $\alpha$ is a cocycle representing a class in $H^1(k, \Aut(\Sigma))$.
In other words, $\alpha$ is a homomorphism
$\Gamma \to \GL(N) \cong \GL_n(\Z) \cong \Aut(\gm^n)$.
Thus, by twisting, we obtain a torus $T = \leftidx{^\alpha}{(\gm^n)}$ 
and a neutral toric $T$-variety $\leftidx{^\alpha}{(X(\Sigma))}$.
All neutral toric varieties can be obtained in this way.

More generally,
suppose $X=\leftidx{^\gamma}{(X(\Sigma))}$ is a twisted form of a split toric variety
for a cocycle $\gamma$ representing a class in
$H^1(k, \gm^n \rtimes \Aut(\Sigma))$.
There is a ``factorization''
$X=\leftidx{^\beta}{(}\leftidx{^\alpha}{(X(\Sigma))})$,
where $\alpha$ represents a class in $H^1(k, \Aut(\Sigma))$
and $\beta$ represents a class in $H^1(k, T)$
where $T := \leftidx{^\alpha}{(\gm^n)}$.
Informally, $\beta$ changes the torus that acts on $X$,
while $\alpha$ changes the torsor of the open orbit in $X$.

Let $X$ be a toric $T$-variety and let $M$ be the Cartier dual of $T$.
Let $\op{Div}_{\overline{T}}(\overline{X})$ be the group of
$\overline{T}$-invariant divisors on $\overline{X}$.
By functoriality, $\Gamma$ acts on $\op{Div}_{\overline{T}}(\overline{X})$
and $\op{Pic}(\overline{X})$.
We have an exact sequence
\begin{displaymath}
  0 \to M \to \op{Div}_{\overline{T}}(\overline{X}) \to \op{Pic}(\overline{X}) \to 0,
\end{displaymath}
just like in the split case.
The K-theory $K_0(\overline{X})$ also has an action of $\Gamma$ by
functoriality.
In fact, from \cite{MerkPan} we see that $K_0(X)$ is a $\Gamma$-lattice
and $K_0(X) \cong K_0(\overline{X})^\Gamma$.

\subsection{Torsors invisible to Brauer invariants}
\label{subsection:icp} 

We recall the set $\ICP(k,G)$ introduced in \cite{BDLM}. This collection serves as the source of torsors which produce arithmetically interesting forms via twisting. 

Let $G$ be a connected reductive algebraic group over $k$, e.g., a torus. The Galois cohomology pointed sets $H^i(k, G)$ are functorial in both $k$ and $G$. We view $H^i(-,G)$ as a functor from the category of field extensions of $k$ to the category of points sets (or groups if $i = 0$, or abelian groups if $G$ is commutative). Following \cite{Skip}, a \emph{normalized cohomological invariant} (of degree 2) is a natural transformation $$\alpha: H^1(-, G) \to H^2(-, S)$$ where $S$ is a commutative linear algebraic group and $\alpha$ takes the distinguished point to 0. We note that since $H^2(k, R_{E/k} \gm) = \op{Br}(E)$ for an \'etale algebra $E$, one recovers Brauer invariants by taking $S = R_{E/k} \gm$.

\begin{definition}
 Let
\begin{displaymath}
  \ICP(k,G) := \bigcap_E \bigcap_\alpha \op{ker} 
  \left( \alpha(k) \colon H^1(k,G) \to \Br(E) \right),
\end{displaymath}
where the intersections run over all \'etale algebras $E$
and all normalized cohomological invariants $\alpha$.
\end{definition}

The set $\ICP(k,G)$ measures torsors which cannot be detected by 
any Brauer invariant. In the case $k$ is a number field, we have the 
following characterization.

\begin{proposition} \label{proposition:icp_tori}
  Let $T$ be a torus over a number field $k$. Then there exists a 
  canonical isomorphism 
  \begin{displaymath}
    \ICP(k,T) \cong \Sha^1(k,T)
  \end{displaymath}
  where $\Sha^1$ is the Tate-Shafarevich group of $T$.
\end{proposition}

\begin{proof}
  This follows from \cite[Theorem 2]{BDLM} since, $\Zhe(k_v, T_v) = \ast$ for any real place $v$. Indeed, any $k_v$-torus is rational and $\Zhe$ is a birational 
  invariant of tori \cite[Proposition 4.2]{BDLM}. 
\end{proof}


\section{Linearizing \'etale-exceptional objects}
\label{section:linearizing_stuff}

In this section we investigate the question of twisting objects and the effect on endomorphism algebras. This involves understanding linearization of objects, especially exceptional ones. These techniques are used in Section~\ref{section:TCI}  to descend exceptional collections on split forms to \'etale-exceptional collections on twisted forms (see Proposition~\ref{prop:cat0_no_pts} and Lemma~\ref{lemma:CT_standard}). 

We begin by recalling the notion of $G$-equivariant structures and establish a relation to linearizations (Prop.~\ref{proposition:linearization_equiv}). We discuss twisting of objects and endomorphism algebras by $G$-torsors (Prop.~\ref{prop:twist_object_twists_end}). For an object $E$, we introduce the sheaf of groups $\widetilde{G}_E$ using the Isom functor (Def.~\ref{defn:Gtilde}). It is then established that $E$ is linearizable if and only if the map $\widetilde{G}_E \to G$ admits a homomorphic section (Lemma~\ref{lem:linearization=section}). We show that for \'etale-exceptional objects $E$ on smooth projective schemes, $E$ is linearizable up to direct sums if $g^* E_{\bar k } \cong E_{\bar k }$ for all $g \in G(\overline k)$ (Prop.~\ref{prop:lin_up_to_sum}). This holds when $G$ is connected (Lemma~\ref{lem:G_conn_E_fixed}).

We utilize the language of fibered categories and stacks following \cite{vistoli_stack, stacks-project}. We let $\textbf{Sch}_k$ denote the category of $k$-schemes.
Suppose $G$ is an algebraic group over $k$, and let $X$ be a $G$-variety. Denote the projection by 
\begin{displaymath}
  \pi : G \times X \to X, 
\end{displaymath}
the action by 
\begin{displaymath}
  \sigma : G \times X \to X,
\end{displaymath}
and the other projection by 
\begin{displaymath}
  p : G \times X \to G. 
\end{displaymath}

We recall the definition of equivariant objects in a fibered category $\mathcal C$ (see, e.g., \cite[Section 3.8]{vistoli_stack}). 
It is phrased in terms of the three maps $
G \times G \times X \to X$ given by 
\begin{itemize}
 \item $\pi_3 \text{ (the third projection) } $ 
\item $\sigma \circ (m \times \op{id}_X) = \sigma \circ (\op{id}_G \times \sigma)$
\item  $\pi \circ (\op{id}_G \times \sigma) = \sigma \circ \pi_{23},$ 
\end{itemize}
where $\pi_{23}$ is the projection onto the second and third components. 

\begin{definition}
  Let $\mathcal C$ be a fibered category over $\textbf{Sch}_k$. An \emph{equivariant object} of $\mathcal C(X)$ is a pair $(E,\phi)$ with $E \in \mathcal C(X)$ and $\phi: \pi^\ast E \to \sigma^\ast E$ an isomorphism satisfying the cocycle condition, given by the commutativity of the diagram 

  \begin{center}
    \begin{equation}\label{diagram:cocycle}
    \begin{tikzpicture}
      \node (p) at (-2,0) {$\pi_3^\ast (E)$};
      \node (b) at (2,0) {$(\sigma \circ \pi_{23})^\ast (E)$};
      \node (a) at (0,-2) {$( \sigma \circ (\op{id}_G \times \sigma))^\ast (E)$};
      \draw[->] (p) -- node[above]{$(m\times \operatorname{id}_X)^\ast \phi$} (b);
      \draw[->] (p) -- node[below left]{$\pi_{23}^\ast \phi$} (a);
      \draw[->] (a) -- node[below right]{$(\operatorname{id}_G \times \sigma)^\ast \phi$} (b);
    \end{tikzpicture}
      \end{equation}
  \end{center}

We denote the fibered category of equivariant objects as $\mathcal{C}_G$.
\end{definition}

\begin{definition}
The \emph{equivariant bounded derived category} $\op{D}^{\op{b}}_G(X)$ is the derived category of the abelian category of equivariant chain complexes. 
\end{definition}

Let $U$ be a right $G$-torsor.
We define the \emph{twist of $X$ by $U$} as the quotient
$\twist{X}{U} := (U \times X)/G$.
We have a diagram
\[
X \xleftarrow{\pi} U \times X \xrightarrow{s} \twist X U,
\]
where $\pi$ is the projection and $s$ is the quotient by the diagonal
$G$-action.

\begin{proposition}
  Let $G$ be a linear algebraic group over a field $k$,
  $X$ a left $G$-variety, and 
  $U$ a right $G$-torsor. Let $\mathcal C$ be an fppf stack over $\mathbf{Sch}_k$. 
  Then the functor 
  \begin{displaymath}
    \Psi_U := (s^\ast)^{-1} \circ \pi^\ast : \mathcal C_G (X) \to \mathcal C \left(\twist X U\right)
  \end{displaymath}
  is well-defined. 
\end{proposition}

\begin{proof}
  Note that we have a commutative diagram 
  \begin{center}
    \begin{tikzpicture}
      \node (a) at (-1.5,1) {$G \times U \times X$};
      \node (b) at (1.5,1) {$U \times X$};
      \node (c) at (-1.5,-1) {$G \times X$};
      \node (d) at (1.5,-1) {$X$};
      \draw[->] (a) -- (b);
      \draw[->] (a) -- node[left]{$1 \times \pi$} (c);
      \draw[->] (c) -- (d);
      \draw[->] (b) -- node[right]{$\pi$} (d);
    \end{tikzpicture}
  \end{center}
  where the horizontal arrows are the actions of $G$. 
  Thus, we have a functor
  \begin{displaymath}
    \pi^\ast : \mathcal C_G(X) \to \mathcal C_G (U \times X).
  \end{displaymath}
  Since $U \times X \to \twist X U$ is a $G$-torsor, we apply 
  \cite[Theorem 4.46]{vistoli_stack} to conclude that 
  pullback by the projection  
$s^\ast : \mathcal C \left(\twist X U \right) \to \mathcal C_G( U \times X)$
  is an equivalence, yielding a well-defined functor $(s^\ast)^{-1} \circ \pi^\ast.$
\end{proof}

The main case of interest is where $\mathcal C$ is the stack 
$\mathcal D^b_{\operatorname{pug}}(X)$ of \cite{Lieblich} when $X$ is assumed to be smooth and projective. 

\begin{definition}
  Let $\mathcal D^b_{\op{pug}}(X)$ denote the stack of \emph{perfect universally gluable objects}, i.e., those perfect objects $E$ of $\rm{D}^{\rm{b}}(X)$ satisfying 
  \begin{displaymath}
    \op{Ext}^{-i}_X(E,E) = 0 
  \end{displaymath}
  for all $i > 0$. This is an Artin stack over $k$ \cite[Main Theorem]{Lieblich}. 
\end{definition}

Recall that for a finite-dimensional $k$-algebra $A$ with a 
$G$-action, the \emph{twist of $A$ by $U$} is the invariant 
algebra 
\begin{displaymath}
  \twist{A}{U} := \left( A \otimes_k k[U] \right)^G.
\end{displaymath}

\begin{proposition} \label{prop:twist_object_twists_end}
  Assume that $X$ is smooth and projective. Then for any 
  $G$-equivariant object $E$ of $\mathcal D^b_{\operatorname{pug}}(X)$ 
  there is a natural isomorphism 
  \begin{displaymath}
    \operatorname{End}_{\twist X U} (\Psi_U E) \cong \twist{\operatorname{End}_X(E)}{U}.
  \end{displaymath}
\end{proposition}

\begin{proof}
  We have $\operatorname{End}_{\twist X U } ( \Psi_U E ) = 
    \left( \operatorname{End}_{U \times X}(p^\ast E) \right)^G.$
  From adjunction, with $p$ flat and affine, we have an isomorphism 
   $\operatorname{Hom}_{U \times X}( p^\ast E, p^\ast E ) \cong 
    \operatorname{Hom}_{X}(E, p_\ast p^\ast E).$
  Viewing $k[U]$ as being pulled back from $\operatorname{Spec}(k)$, the projection formula yields $p_\ast p^\ast E \cong E \otimes_k k[U].$
  Since $E$ is perfect, the natural map 
  \begin{displaymath}
    \op{Hom}_X (E,E) \otimes_k k[U] \to \op{Hom}_X (E, E \otimes_k k[U])
  \end{displaymath}
  is an isomorphism. Thus, 
$\operatorname{End}_{\twist X U } ( \Psi_U E ) \cong  
    \twist{\operatorname{End}_X(E)}{U}. $
  
\end{proof}

\begin{definition}
  Let $X$ be a scheme with an action of algebraic group $G$. We say an object $E$ of $\op{D^b}(X)$ 
  is \emph{linearizable up to sums} if there is some $r \geq 1$ such that
  $E^{\oplus r}$ lies in the essential image of the forgetful functor
  \begin{displaymath}
    \op{D^b_{\emph{G}}}(X) \to \op{D^b}(X).
  \end{displaymath}
  If $r=1$, we say that $E$ is \emph{linearizable}. A choice of lift 
  of $E$ to $\op{D^b_{\emph{G}}}(X)$ shall be called a \emph{linearization}
  of $E$. 
\end{definition}

\begin{proposition} \label{proposition:linearization_equiv}
  Assume $X$ is smooth and projective. For an object $E$ of 
  $\mathcal D^b_{\op{pug}}(X)$, a linearization of $E$ is equivalent 
  to a $G$-equivariant structure for $E$, i.e., an isomorphism 
  satisfying the cocycle condition \ref{diagram:cocycle}.
\end{proposition}

\begin{proof}
  The category $\op{D^b_{\emph{G}}}(X)$ has as its objects complexes 
  with choices of equivariant structure on the complex. This provides the 
  equivariant structure for $E$ in $\mathcal D^b_{\op{pug}}(X)$. 

  The other direction is a consequence of \cite[Th\'eor\`eme 3.2.4]
  {Perverse_Sheaves}, see \cite[proof of Lemma 2.2]{Polishchuk_unity}. 
\end{proof}

Our first goal is to construct a canonical group sheaf which controls the existence of linearizations for any object of a stack. This is accomplished using the Isom-functor. 

\begin{definition}\label{defn:Gtilde}
  Let $\mathcal C$ be a fibered category over $\mathbf{Sch}_k$, 
  $f : X \to Y$ a morphism of $k$-schemes, and $E,F$ objects 
  of $\mathcal C(X)$. The \emph{Hom-functor} associated to this data 
  has 
  \begin{displaymath}
    \operatorname{Hom}_f(E,F)(T) := 
    \lbrace g : T \to Y \ , \alpha: E_T\to F_T \rbrace 
  \end{displaymath}
  for a test $k$-scheme $T$. Here $E_T,F_T$ are the pullbacks to $X_T$
  via the map coming from the Cartesian diagram
  \begin{center}
    \begin{tikzpicture}
      \node (a) at (-1,1) {$X_T$};
      \node (b) at (1,1) {$X$};
      \node (c) at (-1,-1) {$T$};
      \node (d) at (1,-1) {$Y$};
      \draw[->] (a) -- (b);
      \draw[->] (a) -- (c);
      \draw[->] (b) -- node[right] {$f$} (d);
      \draw[->] (c) -- node[below] {$g$} (d);
    \end{tikzpicture}
  \end{center}
  The \emph{Isom-functor} $\operatorname{Isom}_f(E,F)$ is the 
  subfunctor of $\operatorname{Hom}_f(E,F)$ where $\alpha$ is 
  required to be an isomorphism.
  In the special case where $f = p : G \times X \to G$ with $G$ acting on $X$, we use the notation 
  \begin{displaymath}
    \widetilde{G}_E := \operatorname{Isom}_p(\sigma^\ast E, \pi^\ast E).
  \end{displaymath}
\end{definition}

\begin{remark} \label{remark:geo_pts_GE}
  In Proposition~\ref{proposition:isom_group}, we equip $\widetilde{G}_E$ with a group structure and produce a homomorphism (of sheaves of groups) to $G$. Before giving the definition for general $T$-points, we describe here the operations on geometric points. 
  If $g \in G(\overline{k})$, then a $\overline{k}$-point over $g$ is given by 
  an isomorphism on $X$
  \begin{displaymath}
    \alpha : g^\ast E_{\bar{k}} \to E_{\bar{k}}.  
  \end{displaymath}
  Composition for $\overline{k}$-points reduces to 
  \begin{displaymath}
    (g_1,\alpha_1) \cdot (g_2,\alpha_2) = (g_1g_2, \alpha_2 \circ 
    g_2^\ast \alpha_1).
  \end{displaymath}
  The identity is just the identity map of $E$ over $e \in G(\bar{k})$. 
  Finally, the inverse is given by 
  \begin{displaymath}
    (g,\alpha)^{-1} = (g^{-1}, \left( g^{-1} \right)^\ast \alpha^{-1}). 
  \end{displaymath}
\end{remark}

\begin{definition}
  Let $\mathcal C$ be a fibered category over $\mathbf{Sch}_k$ and $E \in \mathcal C(X)$. The functor $\op{Aut}(E)$ has as $T$-points $\op{Aut}_{T \times X}(E_{T \times X})$ for a test scheme $T$ and projection $q : T \times X \to X$. 
\end{definition}

The following proposition establishes the group structure on $\widetilde
G _E$, the first step necessary in recognizing its role in controlling
linearizations of $E$. The proof is straight-forward but technical. As
such, we have relegated it to Appendix~\ref{appendix:proof_group_structure}. 

\begin{proposition} \label{proposition:isom_group}
  Let $\mathcal C$ be a fibered category over $\mathbf{Sch}_k$
  and $E \in \mathcal C(X)$. The Isom-functor 
  $\widetilde{G}_E$ admits a group structure. Furthermore, there are natural transformations 
  \begin{displaymath}
    \op{Aut}(E) \to \widetilde{G}_E 
    \to G 
  \end{displaymath}
  which are homomorphisms of sheaves of groups. When restricted to geometric points, the group operations are as claimed in Remark~\ref{remark:geo_pts_GE}.
\end{proposition}

The following result makes plain exactly how $\widetilde{G}_E$ controls linearizations of $E$. 

\begin{lemma} \label{lem:linearization=section}
  The object $E$ admits a $G$-linearization if and only if the 
  homomorphism $\widetilde{G}_E \to G$ admits a section which is 
  also a homomorphism.  
\end{lemma}

\begin{proof}
  A $G$-point of $\phi : G \to \widetilde{G}_E$ over the 
  identity $1 : G \to G$ is an isomorphism $\alpha : 
  \sigma^\ast E \to \pi^\ast E$. The map $\phi$ is a homomorphism 
  if and only if $\alpha$ satisfies the cocycle condition. 
\end{proof}



\begin{remark}[Representability]
The question of representability of $\widetilde{G}_E$ by an affine group scheme is treated in more detail in Appendix~\ref{appendix:representability}. We believe these considerations are of independent interest, although they are not used directly in our main results. In summary, a result of \cite{stacks-project} is enough to establish
representability of $\widetilde{G}_E$ in the case of a coherent sheaf (Proposition~\ref{prop:rep_GE_coh}). For exceptional objects (or more generally, perfect universally gluable objects) on smooth projective schemes, results of Lieblich \cite{Lieblich} help in showing that it suffices to check that $g^*E \cong E$ for any geometric point $g \in G(\overline k)$ (Proposition~\ref{prop:pug_rep}).
\end{remark}

\begin{lemma} \label{lem:oplus_pushout}
  Assume that $E$ is indecomposable and $\op{Aut}(E)$ is abelian. 
  Then there is a natural isomorphism 
  \begin{displaymath}
    \left( \operatorname{Aut}(E^{\oplus r}) \times \widetilde{G}_E \right) 
    / \op{Aut}(E) \cong \widetilde{G}_{E^{\oplus r}}. 
  \end{displaymath}
\end{lemma}

\begin{proof}
  Denote
    $H := (\op{Aut}(E^{\oplus r}) \times \widetilde{G}_E)/\op{Aut}(E).$ Here $\op{Aut}(E)$ is included in the first factor via the diagonal embedding $\op{Aut}(E) \to \op{Aut}(E^{\oplus r}).$
  We have an extension 
  \begin{displaymath}
    1 \to \op{Aut}(E^{\oplus r}) \to H \to G \to 1
  \end{displaymath}
  which fits into a commutative diagram 
  \begin{center}
    \begin{tikzpicture}
      \node (a) at (-4,1) {$1$};
      \node (b) at (-2,1) {$\op{Aut}(E^{\oplus r})$}; 
      \node (c) at (0,1) {$H$};
      \node (d) at (2,1) {$G$};
      \node (e) at (4,1) {$1$};
      \node (f) at (-4,-1) {$1$};
      \node (g) at (-2,-1) {$\op{Aut}(E^{\oplus r})$};
      \node (h) at (0,-1) {$\widetilde{G}_{E^{\oplus r}}$};
      \node (i) at (2,-1) {$G$};
      \node (j) at (4,-1) {$1$};
      \draw[->] (a) -- (b);
      \draw[->] (b) -- (c);
      \draw[->] (c) -- (d);
      \draw[->] (d) -- (e);
      \draw[->] (b) -- (g);
      \draw[->] (c) -- (h);
      \draw[->] (d) -- (i);
      \draw[->] (f) -- (g);
      \draw[->] (g) -- (h);
      \draw[->] (h) -- (i);
      \draw[->] (i) -- (j);
    \end{tikzpicture}
  \end{center}
  with the two outer vertical maps being isomorphisms. Thus, the middle 
  one is also an isomorphism. 
\end{proof}

\begin{proposition} \label{prop:lin_up_to_sum}
  Let $X$ be smooth and projective. Assume that $E$ is 
  \'etale-exceptional and $g^\ast E_{\bar{k}} \cong E_{\bar{k}}$ for each $g \in 
  G(\overline{k})$. There is some $r \geq 1$ such 
  that $E^{\oplus r}$ is $G$-linearizable.
\end{proposition}

\begin{proof}
  To split the extension, 
  \begin{displaymath}
    1 \to \op{Aut}(E^{\oplus r}) \to (\op{Aut}(E^{\oplus r}) \times \widetilde{G}_E)/\op{Aut}(E) \to G 
    \to 1 
  \end{displaymath}
  it suffices to locate a map 
 $\widetilde{G}_E \to \op{Aut}(E^{\oplus r})$
  filling in the diagram 
  \begin{center}
    \begin{tikzpicture}
      \node (a1) at (-2,1) {$\op{Aut}(E)$};
      \node (g) at (0,1) {$\widetilde{G}_E$};
      \node (ar) at (-2,-1) {$\op{Aut}(E^{\oplus r})$};
      \draw[->] (a1) -- (g);
      \draw[->] (a1) -- (ar);
      \draw[->,dashed] (g) -- (ar);
    \end{tikzpicture}
  \end{center}
  From Lemma~\ref{lem:oplus_pushout} and Lemma
  ~\ref{lem:linearization=section}, we see that filling the above 
  diagram also provides a $G$-linearization for $E^{\oplus r}$. 
If $E$ is \'etale-exceptional have $\op{Aut}(E) = R_{L/k} \mathbb{G}_m$
  for some extension $L/k$. We are looking to fill 
  \begin{center}
    \begin{tikzpicture}
      \node (a1) at (-2,1) {$R_{L/k}\mathbb{G}_m$};
      \node (g) at (0,1) {$\widetilde{G}_E$};
      \node (ar) at (-2,-1) {$R_{L/k}\op{GL}_r$};
      \draw[->] (a1) -- (g);
      \draw[->] (a1) -- (ar);
      \draw[->,dashed] (g) -- (ar);
    \end{tikzpicture}
  \end{center}
  From functoriality of Weil restriction, we reduce to filling in 
  \begin{center}
    \begin{tikzpicture}
      \node (a1) at (-2,1) {$\mathbb{G}_{m,L}$};
      \node (g) at (0,1) {$\left( \widetilde{G}_E \right)_L$};
      \node (ar) at (-2,-1) {$\op{GL}_{r,L}$};
      \draw[->] (a1) -- (g);
      \draw[->] (a1) -- (ar);
      \draw[->,dashed] (g) -- (ar);
    \end{tikzpicture}
  \end{center}
  Since the map $\mathbb{G}_{m,L} \to \left( \widetilde{G}_E \right)_L$ 
  is an embedding, we can find an element $z \in L[\widetilde{G}_E]$ 
  of weight $1$ with respect to the induced $\mathbb{G}_{m,L}$-action. 
  The element $z$ lies in some finite-dimensional 
  $\left( \widetilde{G}_E \right)_L$-representation, which gives us 
  our desired map 
    $\left( \widetilde{G}_E \right)_L \to \op{GL}_{r,L}.$ 
\end{proof}

\begin{lemma} \label{lem:G_conn_E_fixed}
  Let $E \in \mathcal D^b_{\op{pug}}(X)$.   
  If $G$ is connected and $\op{Ext}^1_X (E,E) = 0$, then 
  $g^\ast E_{\bar{k}} \cong E_{\bar{k}}$ for any $g \in G(\overline{k})$.  
\end{lemma}

\begin{proof}
  This is contained in the proof of \cite[Lemma 2.2]{Polishchuk_unity}. We recap it for convenience. We have a map $G \to \mathcal D^b_{\op{pug}}(X)$ corresponding to the sheaf $\sigma^\ast E$ on $G \times X$. The tangent space at $E$ in 
  $\mathcal D^b_{\op{pug}}(X)$ is $\op{Ext}^1_X (E,E)$ \cite[Theorem 3.1.1]{Lieblich}. Since we assumed 
  that this is zero, the map $G^0 \to \mathcal D^b_{\op{pug}}(X)$ is 
  constant. The conclusion follows from the assumption that $G^0 = G$.
\end{proof}

\section{\'Etale-exceptional objects on neutral toric varieties} 
\label{section:TCI}

In this section, we identify a particular class of exceptional objects 
on split toric varieties which descend to \'etale-exceptional objects 
on any neutral model. This is accomplished using the linearization and descent results of Section~\ref{section:linearizing_stuff}. We check that Castravet and Tevelev's 
collection given in Subsection~\ref{subsection:a3} is of this particular form. 

\subsection{TCI-type collections on toric varieties}

Let $X(\Sigma)$ be a split smooth projective toric variety associated to a fan $\Sigma$. Let $R$ denote the Cox ring of $X(\Sigma)$, so that
\begin{displaymath}
  R \cong k[x_\rho \mid \rho \in \Sigma(1)]. 
\end{displaymath}
The Cox ring is graded by $\op{Pic}(X(\Sigma))$, where the weight of $x_\rho$ is $\mathcal O(D_\rho) \in \op{Pic}(X(\Sigma))$. We will identify weights with elements of $\op{Pic}(X(\Sigma))$. 

The finite group $\op{Aut}(\Sigma)$ acts via homogeneous automorphisms on $R$. For a weight $\chi$ and graded $R$-module $M$, we let $M(\chi)$ be the graded $R$-module with $M(\chi)_\psi = M_{\chi + \psi}$. 

Recall that $X(\Sigma) \cong U/\cd{\op{Pic}(X(\Sigma))}$ for a quasi-affine open subset $U$ of $\op{Spec} R$. As such, we have a restriction functor 
\begin{displaymath}
  j^\ast : \op{D}^{\op{b}}_{\op{Pic}}(\mathbb{A}^{\Sigma(1)}) \to \op{D^b}(X). 
\end{displaymath} 

\begin{definition}\label{def:TCI}
 We say $X(\Sigma)$ has an exceptional collection of \emph{toric complete intersection type} or \emph{TCI-type} if there exists a set of graded $R$-modules $F_1,\ldots,F_t$ such that 
 \begin{itemize}
   \item for each $1 \leq s \leq t$
 \begin{displaymath}
   F_s = R(\chi_s)/(x_l \mid l \in I_s) 
 \end{displaymath}
 for some $\chi_s \in \op{Pic}(X(\Sigma))$ and $I_s \subseteq \Sigma(1)$,
 \item the set $F_1,\ldots,F_t$ is $\op{Aut}(\Sigma)$-stable, and
 \item the set $j^\ast F_1,\ldots, j^\ast F_t$ forms a $k$-exceptional collection of $\op{D^b}(X(\Sigma))$. 
 \end{itemize}
\end{definition}

\begin{proposition} \label{proposition:etale_collection}
  Assume $L/k$ is Galois. Let $X(\Sigma)$ be a split smooth projective 
  toric variety over $L$ with fan $\Sigma$ and $X$ a neutral smooth 
  projective toric $T$-variety over $k$ such that 
  $X(\Sigma) \cong X_L$. If $X(\Sigma)$ 
  possesses a full exceptional collection of TCI-type, then $X$ 
  possesses a full \'etale exceptional collection. 
\end{proposition}

\begin{proof}
  The action of $G := \op{Gal}(L/k)$ induces an action on the fan 
  $\Sigma$ and hence a homomorphism $G \to \op{Aut}(\Sigma)$. Through 
  this homomorphism, we have actions of $G$ on both 
  the Picard group $\op{Pic}(X(\Sigma))$ and on the set of rays $\Sigma(1)$. 

  The Galois group $G$ acts on the spectrum of the Cox ring 
  $\op{Spec} R = \mathbb{A}_L^{\Sigma(1)}$ coming from extending 
  the action of $G$ on $\Sigma(1)$ linearly over $k$ and 
  then skew-linearly, via $G$, over $L/k$. 

  For a graded $R$-module of the form $R(\chi)/(x_l \mid l \in I)$ 
  with $I \subseteq \Sigma(1)$, we have a canonical map  
  \begin{displaymath}
    R (g \cdot \chi) \to g \cdot \left( R(\chi)/(x_l \mid l \in I)\right)
  \end{displaymath}
  given by $1 \mapsto 1$. This induces an isomorphism
  \begin{displaymath}
   \sigma_g : R (g \cdot \chi)/(x_l \mid l \in g \cdot I) 
    \to g \cdot \left( R(\chi)/(x_l \mid l \in I)\right).
  \end{displaymath} 
  If $g_1 \cdot \chi = g_2 \cdot \chi$ and $g_1 \cdot I = g_2 \cdot I$, 
  then we have equality 
  \begin{displaymath}
    R (g_1 \cdot \chi)/(x_l \mid l \in g \cdot I) = 
    R (g_2 \cdot \chi)/(x_l \mid l \in g \cdot I). 
  \end{displaymath}
  Thus, we have an isomorphism of graded modules
  \begin{displaymath}
    \sigma_{g_2} \sigma_{g_1}^{-1} : 
    g_1 \cdot \left( R(\chi)/(x_l \mid l \in I)\right) \to 
    g_2 \cdot \left( R(\chi)/(x_l \mid l \in I)\right). 
  \end{displaymath}
  Conversely, if there is an isomorphism of graded modules
  \begin{displaymath}
    g_1 \cdot \left( R(\chi)/(x_l \mid l \in I)\right) \cong  
    g_2 \cdot \left( R(\chi)/(x_l \mid l \in I)\right),
  \end{displaymath}
  then we must necessarily have $g_1 \chi = g_2 \chi$ and 
  $g_1 \cdot I = g_2 \cdot I$. 

  We can partition $F_1,\ldots,F_l$ into its orbits, up to 
  isomorphism, under the action of $G$. It suffices to check that 
  the sum of objects in an orbit descends to to an \'etale 
  exceptional object. We may assume, after relabeling, that 
  $F_1, \ldots, F_l$ is an orbit. Write 
  \begin{displaymath}
    F := F_1 = R(\chi)/(x_l \mid l \in I)
  \end{displaymath}
  and let $H$ be the subgroup of $G$ stablizing both $\chi$ and 
  $I$. Note that $H$ can also be described as the stabilizer 
  of $F_1$ up to isomorphism. Let $g_1,\ldots,g_l$ be a choice 
  of representatives for $G/H$ with $g_1 = 1$. Let 
  \begin{displaymath}
    \tau : G \to S_l = \op{Aut}(G/H).
  \end{displaymath}
  denote the permutation representation furnished by the left 
  action of $G$ on $G/H$. After relabeling, we may write 
  \begin{displaymath}
    F_j = g_j \cdot F. 
  \end{displaymath}
  Then,
  \begin{displaymath}
    g \cdot F_j = g g_j \cdot F = g_{\tau_g(j)} h^g_j \cdot F
  \end{displaymath}
  for a unique $h^g_j \in H$. We have isomorphism 
  \begin{displaymath}
    \psi_g := \oplus g_{\tau_g(j)} \sigma^{-1}_{h^g_j} : 
    g \cdot \left( \bigoplus F_i \right) \to \bigoplus F_i
  \end{displaymath}
  which despite its cumbersome notation is the map that extends 
  the function assigning $1 \in g \cdot F_j$ to $1 \in 
  F_{\tau_g(j)}$. From this description, we see that 
  \begin{displaymath}
    \phi_{g_1} \phi_{g_2} = \phi_{g_1 g_2}.
  \end{displaymath}
  Thus, $\phi$ provides an equivariant structure for $\oplus F_i$. 
  From Galois descent, there exists a module $E$ with $E_L = 
  \oplus F_i$ and with $\op{End}_X(j^\ast E)$ being the 
  field extension determined by the $G$-set $\tau$. 
\end{proof}

\begin{remark}
 Proposition~\ref{proposition:etale_collection} makes clear the difference between a full \'etale-exceptional collection of TCI-type on $X$ and a full $\op{Aut}(\Sigma)$-stable exceptional collection consisting of restrictions of line bundles to intersections of toric divisors. 
 
 Given an object of the form $L|_{D_1 \cap \cdots D_t}$ on $\overline{X}$, we can lift it to $R(\chi)/(x_1,\ldots,x_t)$. Let $H \leq \op{Aut}(\Sigma)$ be the stabilizer of the subset $\{1,\ldots,t\} \subset \Sigma(1)$. Then for each $h \in H$, $h \cdot \chi = \chi + \chi_h$. This gives a class $(\chi_h) \in H^1(H, \op{Ker} i^\ast)$, where $i^\ast : \op{Pic} (\overline{X}) \to \op{Pic}(D_1 \cap \cdots \cap D_t)$ is the restriction map. We can promote an $\op{Aut}(\Sigma)$-stable collection on $\overline{X}$ to an \'etale-exceptional collection of TCI-type on $X$ if and only if $(\chi_h) = 0$ for all objects. 
\end{remark}

The previous proposition produces a full \'etale exceptional collection on the neutral form
once we locate a TCI-type collection on the split form. We can 
leverage our knowledge of $\ICP(k,T)$ to transport this to an 
\'etale collection on forms without $k$-points. First, we record 
that \'etale exceptional objects linearizable up to sums provide 
normalized cohomological invariants. 

\begin{proposition} \label{proposition:exptoinv}
 Assume that $X$ is a smooth and projective variety with an action of a 
 linear algebraic group $G$. Let $E$ be an \'etale-exceptional object 
 in $\rm{D^b}(X)$ with $E^{\oplus r}$ linearizable for $r \geq 1$. Set $L = \op{End}_X(E)$. 
 The map 
 \begin{align*}
  \varphi_E: H^1(-,G) & \to H^2(-,R_{L/k} \mathbb{G}_m) \\
  U & \mapsto \left[ \op{End}_{\twist{X}{U}}(\Psi_U(E^{\oplus r})) \right]
 \end{align*}
 is a degree $2$ normalized cohomological invariant. 
\end{proposition}

\begin{proof}
 The map is clearly a natural transformation of functors. Since 
 we assumed that $E$ is \'etale-exceptional, it is also normalized. 
\end{proof}

\begin{proposition} \label{prop:cat0_no_pts}
  Let $X(\Sigma)$ be a split smooth projective toric variety over a 
  field $k$ possessing a full $k$-exceptional collection of TCI-type.
  Suppose there exists a class in $H^1(k,\Aut(\Sigma))$ such that
  the corresponding torus $T$ satisfies $\ICP(k,T) \ne \ast$.
  Then there exists a $k$-form of $X(\Sigma)$ with a full 
  \'etale exceptional collection, but no rational points.
\end{proposition}

\begin{proof}
  Let $X$ be the neutral form of $X(\Sigma)$ corresponding to $T$.
  By Proposition~\ref{proposition:etale_collection},
  $X$ has a full \'etale exceptional collection. Since $T$ is connected, 
  we can apply Lemma~\ref{lem:G_conn_E_fixed} to see that 
  each object in the collection is fixed by all $t \in T(\bar{k})$. 
  Using Proposition~\ref{prop:lin_up_to_sum}, we know that each 
  object in the collection admits a $T$-linearization up to sums. 
  Let $U$ be a non-trivial $T$-torsor in $\ICP(k,T)$.
  Since $U$ is non-trivial, the twist $\twist{X}{U}$ has no rational points
  by Proposition~\ref{proposition:rational_points}.
  By Proposition~\ref{prop:twist_object_twists_end},
  the twist $\twist{X}{U}$ has an exceptional collection.
  Moreover, since $U \in \ICP(k,T)$, this collection is \'etale exceptional by the observation in Proposition~\ref{proposition:exptoinv}. 
\end{proof}

\subsection{Existence of a full TCI-type collection}
\label{subsection:CT_TCI} 

In view of Proposition~\ref{prop:cat0_no_pts},
the only remaining obstacle to the proof of Theorem~\ref{theorem:BB_no}
is to find a torus $T$ such that 
\begin{itemize}
  \item $T$ admits a smooth compactification as a toric variety $X$ 
  with $X_L$ possessing a full exceptional collection of TCI-type and 
  \item $\ICP(k,T)$ is nontrivial. 
\end{itemize}
In this section, we address the first point by verifying that the 
exceptional collection of Castravet and Tevelev is of TCI-type. 

\begin{lemma} \label{lemma:CT_standard}
 Castravet and Tevelev's exceptional collection is of TCI-type.
\end{lemma}

\begin{proof}
 Everything except the torsion block is a line bundle, so we just 
 need to check that this block lifts to a module over the Cox ring 
 in an $\op{Aut}(\Sigma)$-stable fashion. 
 
 A weight is in particular a linear function $\omega_D : \mathbb{Z} A_3 \to \mathbb{Z}$. The set of roots lying in the kernel of $\omega_D$ is a root system of type $A_1 \times A_1$. Hence, we have a map $\pi: X(A_3) \to X(A_1 \times A_1)$. The composition $\pi \circ i: X(A_1 \times A_1) \to X(A_1 \times A_1)$ is the identity \cite[Remark 1.12]{BatBlu}.
 
 The line bundle $\pi^\ast \mathcal O(-1,-1)$ therefore restricts via $i^\ast$ to $\mathcal O(-1,-1)$. A computation identifies 
 \begin{displaymath}
   \pi^\ast \mathcal O(-1,-1) \cong G_2^\vee(D+D^\prime)
 \end{displaymath}
 where $G_2$ (using the notation of \cite{CT}) is $(S_4 \times C_2)$-fixed and $D^\prime$ is the image of $D$ under the nontrivial element of $C_2$. 
 
 Let $\chi_{G_2}, \chi, \chi^\prime$ be characters of $\cd{\op{Pic}(X(A_3))}$ corresponding to $G_2, \mathcal O(D), \mathcal O(D^\prime)$. Then, we can lift $i_\ast \mathcal O(-1,-1)$ to 
 \begin{displaymath}
   R(-\chi_{G_2} + \chi_{D} + \chi_{D^\prime)}/(x_D). 
 \end{displaymath}
 The action of $S_4 \times C_2$ permutes these choices of lifts.  
\end{proof}

\begin{remark}
 One can also geometrically identify the coherent sheaves coming from 
 the torsion part, after descending. Depending on the homomorphism 
 $\op{Gal}(L/k) \to \op{Aut}(\Sigma)$, the divisor, over $k$, is 
 either a $R_{L/k} \mathbb{P}^1_L$ or $\mathbb{P}^1_k \times 
 \mathbb{P}^1_k$. In either case, 
 the corresponding exceptional object comes from pushing forward 
 the line bundle on the divisor which base changes to 
 $\mathcal O(-1,-1)$. 
\end{remark}


\section{Proofs of the main results}\label{section:mainproofs}

\subsection{Orlov's Conjecture for toric varieties}

We first prove Theorem~\ref{theorem:FEC_rat} from
the introduction. This follows from a results of Voskresenskii 
on rationality of certain tori. 

\begin{proof}[Proof of Theorem~\ref{theorem:FEC_rat}]
  If $E_1,\ldots,E_n$ is a (general) exceptional collection, then over $\overline{k}$ we have 
  \begin{displaymath}
     (E_i)_{\overline{k}} = \bigoplus (E^j_i)^{\oplus r_i} , 
  \end{displaymath}
  where the $E^j_i$ are distinct $\overline{k}$-exceptional objects 
  permuted by $\Gamma_k$. If $E_i$ are $k$-exceptional, then $r_i = 1$. 
  Thus, the classes $[E^j_i]$ form a $\Gamma$-fixed basis for $\op{K}_0
  (\overline{X})$ so that $\op{K}_0(\overline{X})$ has a trivial 
  $\Gamma$-action.
  
  Since we have a surjective map $\op{det} : \op{K}_0(\overline{X}) \to
  \op{Pic}(\overline{X})$ with $\op{K}_0(\overline{X})$ carrying a trivial
  $\Gamma$-action, the module $\op{Pic}(\overline{X})$ has trivial
  $\Gamma$-action.
  We have a short exact sequence of $\Gamma$-lattices
   \begin{displaymath}
     0 \to \widehat{T} \to \op{Div}(\overline{X}) \to \op{Pic}(\overline{X}) \to 0,
   \end{displaymath}
  where $\op{Div}(\overline{X})$ is permutation and $\op{Pic}(\overline{X})$ is
  trivial.
  Taking Cartier duals we have an exact sequence of tori
  \begin{displaymath}
     1 \to \mathbb{G}_{m}^{r} \to R_{E/k} \mathbb{G}_m \to T \to 1,
   \end{displaymath}
  for some \'etale algebra $E$.
  From \cite[Theorem 2]{VoskyRatl},
  we conclude that $T$ is rational, and thus so is $X$.
\end{proof} 

\subsection{Non-rational toric varieties with full \'etale exceptional collections}

In this section, we prove Theorem~\ref{theorem:BB_no}. We do this by 
constructing an example that satisfies the requirements of 
Proposition~\ref{proposition:exptoinv}. While our arguments can be 
generalized (with appropriate assumptions) to global fields, we 
stick to working over $\mathbb{Q}$. 

\begin{lemma} \label{lemma:5_29_Q}
  The following facts hold for $\mathbb{Q}(\sqrt{5},\sqrt{29})/\mathbb{Q}$:
  \begin{itemize}
    \item $\op{Gal}(\mathbb{Q}(\sqrt{5},\sqrt{29})/\mathbb{Q}) 
    \cong C_2 \times C_2$ 
    \item $\ICP(\mathbb{Q}, 
    R^{(1)}_{\mathbb{Q}(\sqrt{5},\sqrt{29})/\mathbb{Q}} \mathbb{G}_m)
    = C_2.$ 
  \end{itemize}
\end{lemma}

\begin{proof}
  The first fact is straightforward. For the second, use 
  Proposition~\ref{proposition:icp_tori} and the fact that 
  \begin{displaymath}
    \Sha^1(\mathbb{Q}, 
    R^{(1)}_{\mathbb{Q}(\sqrt{5},\sqrt{29})/\mathbb{Q}} \mathbb{G}_m)
    = C_2,
  \end{displaymath}
  as calculated in \cite[Example 11.6.3.2]{Vosky}. 
\end{proof}

For convenience, we set 
$G := \op{Gal}(\mathbb{Q}(\sqrt{5},\sqrt{29})/\mathbb{Q})$. 
We let $\Sigma$ denote the fan of $X(A_3)$ in $N_\mathbb{R}$. 

\begin{lemma} \label{lemma:gal_action_on_fan}
  There is a homomorphism 
\begin{displaymath}
  \phi : G \to \op{Aut}(\Sigma)
\end{displaymath} 
  such that the Cartier dual to the 
  $G$-module $M = N^{\vee}$ is the norm-one torus 
  $R^{(1)}_{\mathbb{Q}(\sqrt{5},\sqrt{29})/\mathbb{Q}} \mathbb{G}_m$. 
\end{lemma}

\begin{proof}
 Let $\sigma$ be the composition of the outer involution $(-1)$ 
 and the permutation $(12)$ in 
 $\op{Aut}(\Sigma)$ and let $\tau$ be the composition of $(-1)$ 
 and $(03)$ in $\op{Aut}(\Sigma)$. In the basis $e_0,e_1,e_2$ of 
 $N$, $\sigma$ and $\tau$ are represented by the matrices
 \begin{displaymath}
  \begin{pmatrix}
    -1 & 0 & 0 \\
    0 & 0 & -1 \\
    0 & -1 & 0 \\
  \end{pmatrix}
  \begin{pmatrix}
    1 & 0 & 0 \\
    1 & -1 & 0 \\
    1 & 0 & -1 \\
  \end{pmatrix}
\end{displaymath}
respectively. From Lemma~\ref{lemma:aug_ideal_presentation},
we see that $N$ with this $(C_2 \times C_2)$-action is isomorphic to 
the kernel of the augmentation ideal $J$ of 
$\mathbb{Z}[C_2 \times C_2]$. Consequently, the Cartier dual of 
the dual $\Z$-module $M = N^\vee$ is the character group of a norm-
one torus for a biquadratic extension. 
\end{proof} 

The neutral toric variety for the torus 
$R^{(1)}_{\mathbb{Q}(\sqrt{5},\sqrt{29})/\mathbb{Q}} \mathbb{G}_m$, 
coming from $\phi$, will be denoted $X$.

Theorem~~\ref{theorem:Kuznetsov_no} now follows from the following:

\begin{theorem} \label{theorem:non_retract_rat_full_etale}
  The variety $X$ is smooth, geometrically rational, and not
  retract-rational. It possesses a full \'etale exceptional collection whose objects are coherent sheaves.
\end{theorem}

\begin{proof}
  Smoothness and geometric irreducibility are clear. No norm-one tori for 
  biquadratic extensions are ever retract rational 
  \cite[Example 11.6.3.1]{Vosky}. The fact we have a
  full \'etale exceptional collection follows from Proposition
  ~\ref{proposition:etale_collection} and Lemma
  ~\ref{lemma:CT_standard}. 
\end{proof}

\begin{remark}
  Theorem~\ref{theorem:non_retract_rat_full_etale} holds if we 
  start with any biquadratic extension $K/k$ in place of 
  $\mathbb{Q}(\sqrt{5},\sqrt{29})/\mathbb{Q}$. 
\end{remark}

We are finally in a position to prove Theorem~\ref{theorem:BB_no}.

\begin{proof}[Proof of Theorem~\ref{theorem:BB_no}]
We simply need to verify the conditions of
Proposition~\ref{prop:cat0_no_pts} for $X$ and $T$. 
The existence of a full $k$-exceptional collection of TCI-type
follows from Lemma~\ref{lemma:CT_standard}.
The non-triviality of $\ICP(k,T)$ follows from
Lemma~\ref{lemma:5_29_Q}.
\end{proof}

\begin{remark}
 The arguments of Theorem~\ref{theorem:BB_no} go through if we 
 replace $\mathbb{Q}(\sqrt{5},\sqrt{29})/\mathbb{Q}$ by a 
 biquadratic extension $K/k$ with 
 $\ICP(k,R^{(1)}_{K/k} \mathbb{G}_m)$ nontrivial. We leave it 
 to the interested reader to show that one can always find such a field
 $K$ for any global field $k$. 
\end{remark}

The following variant of Orlov's Conjecture is suggested by the results
of this paper:

\begin{conjecture}
  If $X$ is a variety over field $k$ possessing a full \'etale-exceptional 
  collection and a $k$-point then $X$ is $k$-unirational. 
\end{conjecture}

Over a closed field, this is a weakening of Orlov's Conjecture.
The conjecture holds for toric varieties since all toric varieties
with a rational point are unirational.


\appendix

\section{Proof of Proposition~\ref{proposition:isom_group}}\label{appendix:proof_group_structure}

\begin{proof}[Proof of Proposition~\ref{proposition:isom_group}]
  Using \cite[Theorem 3.45]{vistoli_stack}, we can assume that 
  $\mathcal C$ is split. 
We first describe the natural transformations. Let 
  $q : T \times X \to X$ be the projection, and assume we have 
  an isomorphism $\alpha : q^\ast E \to q^\ast E$. Set 
  $g : T \to \operatorname{Spec} k \overset{e}{\to} G$. 
  Then $(g,\alpha)$ is a $T$-point of 
  $\widetilde{G}_E$ above $g$.
  The other natural transformation $\widetilde{G}_E \to G$ is given by $(g,\alpha)  \mapsto g. $
  
Let us define the group structure. Given $(g_1,\alpha_1), 
  (g_2, \alpha_2)$ in 
  $\widetilde{G}_E(T)$, we first 
  set $g: T \to G$ to be the composition
  \begin{displaymath}
    T \overset{\Delta}{\longrightarrow} T \times T \overset{g_1 \times g_2}{\longrightarrow}
     G \times G \overset{m}{\longrightarrow} G 
  \end{displaymath}
  where $m : G \times G \to G$ is the group operation on $G$.
  The corresponding Cartesian diagram can be factored as a sequence 
  of Cartesian diagrams as below. 
  \begin{center}
    \begin{tikzpicture}
      \node (a) at (-5,1) {$T \times X$};
      \node (b) at (-2,1) {$T \times T \times X$};
      \node (c) at (2,1) {$G \times G \times X$};
      \node (d) at (5,1) {$G \times X$};
      \node (e) at (-5,-1) {$T$};
      \node (f) at (-2,-1) {$T \times T$};
      \node (g) at (2,-1) {$G \times G$};
      \node (h) at (5,-1) {$G$};
      \draw[->] (a) -- node[above] {$\Delta \times 1$} (b);
      \draw[->] (a) -- (e);
      \draw[->] (e) -- node[above] {$\Delta$} (f);
      \draw[->] (b) -- node[above] {$g_1 \times g_2 \times 1$} (c);
      \draw[->] (b) -- (f);
      \draw[->] (f) -- node[above] {$g_1 \times g_2$} (g);
      \draw[->] (c) -- node[above] {$m \times 1$} (d);
      \draw[->] (c) -- node[right] {$p_{12}$} (g);
      \draw[->] (g) -- node[above] {$m$} (h);
      \draw[->] (d) -- node[right] {$p$} (h);
    \end{tikzpicture}
  \end{center}
  
  We have a commutative diagram 
  \begin{center}
    \begin{tikzpicture}
      \node (a) at (-4,1) {$T \times T \times X$};
      \node (b) at (0,1) {$G \times G \times X$};
      \node (c) at (4,1) {$G \times X$};
      \node (d) at (-2,-1) {$T \times G \times X$};
      \node (e) at (2,-1) {$T \times X$};
      \draw[->] (a) -- node[above] {$g_1 \times g_2 \times 1$} (b);
      \draw[->] (a) -- node[left=2em,below] {$1 \times g_2 \times 1$} (d);
      \draw[->] (b) -- node[above] {$1 \times \sigma$} (c);
      \draw[->] (d) -- node[below] {$1 \times \sigma$} (e);
      \draw[->] (e) -- node[right=1em,below] {$g_1 \times 1$} (c);
    \end{tikzpicture}
  \end{center}
  Thus, we have equalities
  \begin{gather*}
    (g_1 \times g_2 \times 1)^\ast (1\times \sigma)^\ast \sigma^\ast E 
     = (1 \times g_2 \times 1)^\ast (1 \times \sigma)^\ast
    (g_1 \times 1)^\ast \sigma^\ast E \\
    (g_1 \times g_2 \times 1)^\ast (1\times \sigma)^\ast \pi^\ast E 
    = (1 \times g_2 \times 1)^\ast (1 \times \sigma)^\ast 
    (g_1 \times 1)^\ast \pi^\ast E. 
  \end{gather*}
  Using $\alpha_1$, we have an isomorphism 
  \begin{displaymath}
    \tilde{\alpha}_1 := 
    (1 \times g_2 \times 1)^\ast (1 \times \sigma)^\ast \alpha_1 
    : (g_1 \times g_2 \times 1)^\ast (1\times \sigma)^\ast \sigma^\ast E 
    \to (g_1 \times g_2 \times 1)^\ast (1\times \sigma)^\ast \pi^\ast E
  \end{displaymath}

  Since $\pi \circ (1 \times \sigma) = \sigma \circ p_{23}$, we have  
  \begin{displaymath}
    (1\times \sigma)^\ast \pi^\ast E = p_{23}^\ast \sigma^\ast E. 
  \end{displaymath}
  We also have a commutative diagram 
  \begin{center}
    \begin{tikzpicture}
      \node (a) at (-4,1) {$T \times T \times X$};
      \node (b) at (0,1) {$G \times G \times X$};
      \node (c) at (4,1) {$G \times X$};
      \node (d) at (-2,-1) {$G \times T \times X$};
      \node (e) at (2,-1) {$T \times X$};
      \draw[->] (a) -- node[above] {$g_1 \times g_2 \times 1$} (b);
      \draw[->] (a) -- node[left=2em,below] {$g_1 \times 1 \times 1$} (d);
      \draw[->] (b) -- node[above] {$p_{23}$} (c);
      \draw[->] (d) -- node[below] {$p_{23}$} (e);
      \draw[->] (e) -- node[right=1em,below] {$g_2 \times 1$} (c);
    \end{tikzpicture}
  \end{center}
  Thus, we have  
  \begin{gather*}
    (g_1 \times g_2 \times 1)^\ast p_{23}^\ast \sigma^\ast E 
    = (g_1 \times 1 \times 1)^\ast p_{23}^\ast 
    (g_2 \times 1)^\ast \sigma^\ast E \\
    (g_1 \times g_2 \times 1)^\ast p_{23}^\ast \pi^\ast E 
    = (g_1 \times 1 \times 1)^\ast p_{23}^\ast 
    (g_2 \times 1)^\ast \pi^\ast E. 
  \end{gather*}
  Using $\alpha_2$, we have an isomorphism 
  \begin{displaymath}
    \tilde{\alpha}_2 := 
    (g_1 \times 1 \times 1)^\ast p_{23}^\ast \alpha_2 
    : (g_1 \times g_2 \times 1)^\ast p_{23}^\ast \sigma^\ast E 
    \to (g_1 \times g_2 \times 1)^\ast p_{23}^\ast \pi^\ast E
  \end{displaymath}

  Since $\sigma \circ (m \times 1) =  \sigma \circ (1 \times \sigma)$, 
  we have 
  \begin{displaymath}
    (m \times 1)^\ast \sigma^\ast E = (1 \times \sigma)^\ast 
    \sigma^\ast E.
  \end{displaymath}
  Similarly, we have $\pi \circ (m \times 1) = \pi \circ p_{23}$ so 
  \begin{displaymath}
    (m \times 1)^\ast \pi^\ast E = p_{23}^\ast \pi^\ast E. 
  \end{displaymath}

  Thus, we have an isomorphism 
  \begin{displaymath}
    \alpha := (\Delta \times 1)^\ast (\tilde{\alpha}_2 \circ 
    \tilde{\alpha}_1) : (\sigma^\ast E)_T \to (\pi^\ast E)_T
  \end{displaymath}
  over the map $g : T \to G$. This can be simplified by noting that 
  \begin{displaymath}
    p_{23} \circ (g_1 \times 1 \times 1) \circ (\Delta \times 1) = 1
  \end{displaymath}
  so 
  \begin{align*}
    \alpha & = \alpha_2 \circ (\Delta \times 1)^\ast \tilde{\alpha}_1 \\
    & = \alpha_2 \circ (\Delta \times 1)^\ast  
    (1 \times g_2 \times 1)^\ast (1 \times \sigma)^\ast \alpha_1. 
  \end{align*}

  Given a test scheme $T$, the identity in 
  $\widetilde{G}_E(T)$ has the  
  composition $T \to \operatorname{Spec} k \overset{e}{\to} G$ 
  as the first component. Since $\sigma \circ (e \times 1) = 
  \pi \circ (e \times 1)$, we have 
  \begin{displaymath}
    (e \times 1)^\ast \sigma^\ast E = (e \times 1)^\ast \pi^\ast E.
  \end{displaymath}
  Pulling this back via $T \to \op{Spec} k$ gives the isomorphism 
  \begin{displaymath}
    (\sigma^\ast E)_T = (\pi^\ast E)_T. 
  \end{displaymath}

  Let $\iota : G \to G$ denote the inversion in $G$. Given 
  $g : T \to G$ and $\alpha : \sigma^\ast E \to \pi^\ast E$, 
  the map $\iota \circ g : T \to G$ is the first component 
  of the inverse of $(g,\alpha)$. Denote the following composition 
  \begin{displaymath}
    T \times X \overset{\Delta \times 1}{\to} T \times T \times X 
    \overset{1 \times g \times 1}{\to} T \times G \times X 
    \overset{1 \times \iota \times 1 }{\to } T \times G \times X 
    \overset{1 \times \sigma}{\to} T \times X
  \end{displaymath}
  by $\Phi : T \times X \to T \times X$. Then the second component 
  of the inverse is $\Phi^\ast \alpha^{-1}$. We check this indeed 
  gives a map $(\sigma^\ast E)_T \to (\pi^\ast E)_T$ for 
  $\iota \circ g: T \to G$. One can check that 
  \begin{displaymath}
    \sigma \circ (g \times 1) \circ \Phi = \pi \circ (\iota \times 1) 
    \circ (g \times 1)
  \end{displaymath} 
  and 
  \begin{displaymath}
    \pi \circ (g \times 1) \circ \Phi = \sigma \circ (\iota \times 1) 
    \circ (g \times 1)
  \end{displaymath}
  so indeed 
  \begin{displaymath}
    \Phi^\ast \alpha^{-1} : (\sigma^\ast E)_T  \to (\pi^\ast E)_T 
  \end{displaymath}
  for $\iota \circ g : T \to G$. 

  Now we check the axioms of a group. For associativity with 
  $(g_1,\alpha_1), (g_2,\alpha_2), (g_3,\alpha_3)$, we are 
  comparing 
  \begin{displaymath}
    \alpha_3 \circ (\Delta \times 1)^\ast (1 \times g_3 \times 1)^\ast 
    (1 \times \sigma)^\ast \alpha_2 \circ (\Delta \times 1)^\ast 
    (1 \times g_2g_3 \times 1)^\ast (1 \times \sigma)^\ast \alpha_1
  \end{displaymath}
  and 
  \begin{displaymath}
    \alpha_3 \circ (\Delta \times 1)^\ast (1 \times g_3 \times 1)^\ast 
    (1 \times \sigma)^\ast \left( \alpha_2 \circ 
    (\Delta \times 1)^\ast (1 \times g_2 \times 1)^\ast 
     (1 \times \sigma)^\ast \alpha_1 \right).
  \end{displaymath}
  So it suffices to know that  
  \begin{displaymath}
    (1 \times \sigma) \circ (1 \times g_2g_3 \times 1) \circ 
    (\Delta \times 1) = (1 \times \sigma) \circ (1 \times g_2 \times 1) 
    \circ (\Delta \times 1) \circ  (1 \times \sigma) \circ 
    (1 \times g_2 \times 1) \circ (\Delta \times 1)
  \end{displaymath}
  which is easy to see and follows from $\sigma$ being an action. 

  Next, we verify the inverses are indeed inverses. On one side we have 
  \begin{displaymath}
    \alpha \circ (\Delta \times 1)^\ast (1 \times g \times 1)^\ast 
    (1 \times \sigma)^\ast \Phi^\ast \alpha^{-1}.
  \end{displaymath}
  Note that 
  \begin{displaymath}
    \Phi = (1 \times \sigma) \circ (1 \times g^{-1} \times 1) \circ 
    (\Delta \times 1).
  \end{displaymath}
  Thus, 
  \begin{gather*}
    (\Delta \times 1)^\ast (1 \times g \times 1)^\ast 
    (1 \times \sigma)^\ast \Phi^\ast \alpha^{-1} = (\Delta \times 1)^\ast 
    (1 \times g g^{-1} \times 1)^\ast (1 \times \sigma)^\ast \alpha^{-1}
    \\
    = (\Delta \times 1)^\ast (1 \times e \times 1)^\ast 
    (1 \times \sigma)^\ast \alpha^{-1}.
  \end{gather*}
  Since 
  \begin{displaymath}
    (1 \times \sigma) \circ (1 \times e \times 1) \circ (\Delta \times 1) 
    = 1, 
  \end{displaymath}
  we have 
  \begin{displaymath}
    (\Delta \times 1)^\ast (1 \times e \times 1)^\ast 
    (1 \times \sigma)^\ast \alpha^{-1} = \alpha^{-1}
  \end{displaymath}
  and 
  \begin{displaymath}
    \alpha \circ (\Delta \times 1)^\ast (1 \times g \times 1)^\ast 
    (1 \times \sigma)^\ast \Phi^\ast \alpha^{-1}  = 1. 
  \end{displaymath}
  In the other direction, we have to simplify
  \begin{displaymath}
    \Phi^\ast \alpha^{-1} \circ (\Delta \times 1)^\ast 
    (1 \times g^{-1} \times 1)^\ast (1 \times \sigma)^\ast \alpha.
  \end{displaymath}
  Since 
  \begin{displaymath}
    \Phi = (1 \times \sigma) \circ (1 \times g^{-1} \times 1) \circ 
    (\Delta \times 1), 
  \end{displaymath}
  we have 
  \begin{displaymath}
    \Phi^\ast \alpha^{-1} \circ (\Delta \times 1)^\ast 
    (1 \times g^{-1} \times 1)^\ast (1 \times \sigma)^\ast \alpha = 
    \Phi^\ast \alpha^{-1} \circ \Phi^\ast \alpha = 
    \Phi^\ast ( \alpha^{-1} \circ \alpha)  = 1. 
  \end{displaymath}

  Finally, we need to check the identity is indeed the identity. On 
  one side, we simplify
  \begin{displaymath}
    1 \circ (\Delta \times 1)^\ast (1 \times e \times 1)^\ast 
    (1 \times \sigma)^\ast \alpha = 1 \circ \alpha = \alpha
  \end{displaymath}
  using 
  \begin{displaymath}
    (1 \times \sigma) \circ (1 \times e \times 1) \circ (\Delta \times 1) 
    = 1, 
  \end{displaymath}
  On the other side, we have 
  \begin{displaymath}
    \alpha \circ (\Delta \times 1)^\ast (1 \times g \times 1)^\ast 
    (\Delta \times 1)^\ast 1 = \alpha \circ 1 = \alpha 
  \end{displaymath}
  since pullback preserves identity morphisms. 

  It is straightforward to see that the operations defined reduce to what is described in Remark~\ref{remark:geo_pts_GE}.
\end{proof}


\section{Representability Results for \texorpdfstring{$\widetilde G_E$}{GE}}\label{appendix:representability}

For an object $E$, the sheaf of groups $\widetilde{G}_E$ given in Definition~\ref{defn:Gtilde} is representable in many cases of interest. Given an algebraic space $B$, we let $\mathbf{Sch}_B$ denote the category of schemes over $B$.

\begin{proposition} \label{prop:rep_GE_coh}
  Let $X$ be a smooth and projective $k$-scheme, and let $E$ be a coherent 
  sheaf on $X$. Then $\widetilde{G}_E$ is an affine group scheme. 
\end{proposition}

\begin{proof}
  We take $\mathcal F = \sigma^\ast E$, $\mathcal G = \pi^\ast E$, and $f = p : G \times X \to G$ in
  \cite[\href{https://stacks.math.columbia.edu/tag/08JY}{Lemma 08JY}]
  {stacks-project}. From this, we know that $\widetilde{G}_E$ is an algebraic space which is affine over $G$.   
  Since $G$ is affine, 
  $\widetilde{G}_E$ is affine.
\end{proof}

\begin{proposition} \label{prop:rep_h0}
  Let $S$ be a scheme and $B$ an algebraic space over $S$. 
  Let $K$ be a pseudo-coherent object of $\rm{D}$$(B)$. If for all 
  $g : T \to B$ in $\mathbf{Sch}_B$ the cohomology sheaves 
  $H^i(\mathbf{L}g^\ast K) = 0$ for all $i < 0$, then the functor 
  \begin{align*}
     \mathbf{Sch}_B ^{\operatorname{op}} & \to 
    \operatorname{Set} \\
    g & \mapsto H^0(T,\mathbf{L}g^\ast K)
  \end{align*}
  is an affine algebraic space of finite presentation over $B$.
\end{proposition}

\begin{proof}
  This is contained in the statement of 
  \cite[\href{https://stacks.math.columbia.edu/tag/08JX}
  {Lemma 08JX}]{stacks-project}. 
\end{proof}

We will need a vanishing criterion of \cite{Lieblich}. 

\begin{proposition} \label{prop:pug_check}
  Let $E, F$ be $S$-perfect objects of $\op{D^b}(X)$ for $f : X \to S$ 
  a proper flat morphism of finite presentation between 
  locally Noetherian algebraic spaces. Then, 
  \begin{displaymath}
    \mathbf{R}f_{T \ast} \mathbf{R} \mathcal Hom(E_T, F_T)^{<0} = 0
  \end{displaymath}
  for any map $T \to S$ if and only if 
  \begin{displaymath}
    \operatorname{Ext}_{X_s}^i(E_s,F_s) = 0 
  \end{displaymath}
  for any $i < 0$ and any geometric point $s \in S(\overline{k})$. 
\end{proposition}

\begin{proof}
  This is a slight generalization of the statement of 
  \cite[Proposition 2.1.9]{Lieblich} 
  which proves the case $E \neq F$ but only states the results for $E=F$.
\end{proof}

\begin{proposition}\label{prop:pug_rep}
  Let $E \in \mathcal D^b_{\operatorname{pug}}(X)$ and $X$ be smooth and 
  projective. Assume $g^\ast E \cong E$ for any geometric point 
  $g \in G(\overline{k})$. Then $\widetilde{G}_E$ is representable by an affine group
  scheme. 
\end{proposition}

\begin{proof}
 We have 
 \begin{displaymath}
   \operatorname{Hom}((\sigma^\ast E)_T, (\pi^\ast E)_T) \cong 
   H^0(T, \mathbf{R}f_{T \ast} \mathbf{R} \mathcal Hom (
     (\sigma^\ast E)_T, (\pi^\ast E)_T )).
 \end{displaymath}
 From flat base change, we have 
 \begin{displaymath}
  \mathbf{R}f_{T \ast} \mathbf{R} \mathcal Hom (
    \sigma^\ast E, \pi^\ast E) \cong \mathbf{L}g^\ast \mathbf{R}f_\ast 
    \mathcal Hom (\sigma^\ast E, \pi^\ast E).
 \end{displaymath}
 To apply Proposition~\ref{prop:rep_h0}, we need to check that 
 \begin{displaymath}
    \mathbf{R}f_{T \ast} \mathbf{R} \mathcal Hom (
    (\sigma^\ast E)_T, (\pi^\ast E)_T )^{<0}  = 0.
 \end{displaymath} 
 Applying Proposition~\ref{prop:pug_check}, we reduce to checking 
 \begin{displaymath}
   \op{Ext}^i_{\bar{X}}(g^\ast E, E) = 0
 \end{displaymath}
 for all $g \in G(\bar{k})$ and $i < 0$. By assumption, $g^\ast E 
 \cong E$ so 
 \begin{displaymath}
  \op{Ext}^i_{\bar{X}}(g^\ast E, E) = \op{Ext}^i_{\bar{X}}(E, E)
 \end{displaymath}
 which vanishes for $i < 0$ since we assumed $E \in 
 \mathcal D^b_{\operatorname{pug}}(X)$. Thus $\widetilde{G}_E$ is an 
 algebraic space affine over $X$. Arguing as in the proof of 
 Proposition~\ref{prop:rep_GE_coh}, we conclude it is a 
 scheme. 
\end{proof}


\bibliographystyle{halpha-abbrv}
\bibliography{ExcCollToric}

\end{document}